\documentclass[12pt]{article}

\usepackage[left=1in,top=1in,right=1in,bottom=1in,headheight=3ex,headsep=3ex]{geometry}

\usepackage{amsmath}
\usepackage{amssymb}
\usepackage{amsfonts}
\usepackage{amsthm}
\usepackage{xcolor}
\usepackage{enumerate}

\usepackage[all,cmtip]{xy}

\usepackage{fancyhdr}
\theoremstyle{theorem}
\newtheorem{thm}{Theorem}[section]
\newtheorem{prop}[thm]{Proposition}
\newtheorem{lem}[thm]{Lemma}

\newtheorem{conj}[thm]{Conjecture}
\theoremstyle{definition}
\newtheorem{dfn}[thm]{Definition}
\newtheorem{ex}[thm]{Example}
\theoremstyle{remark}
\newtheorem{rmk}{Remark}[section]

\newcommand{\GL}{\operatorname{GL}}
\newcommand{\GSpin}{\operatorname{GSpin}}
\newcommand{\GSp}{\operatorname{GSp}}
\newcommand{\SO}{\operatorname{SO}}
\newcommand{\Gal}{\operatorname{Gal}}
\newcommand{\End}{\operatorname{End}}

\newcommand{\Cl}{\operatorname{Cl}}
\newcommand{\Aut}{\operatorname{Aut}}
\newcommand{\Hom}{\operatorname{Hom}}
\newcommand{\Pic}{\operatorname{Pic}}
\newcommand{\ch}{\operatorname{char}}
\newcommand{\cl}{\operatorname{cl}}
\newcommand{\Trd}{\operatorname{Trd}}
\newcommand{\SH}{\operatorname{Sh}}
\newcommand{\Br}{\operatorname{Br}}
\newcommand{\Nm}{\operatorname{Nm}}
\renewcommand{\sp}{\operatorname{sp}}

\newcommand{\KS}{\mathrm{KS}}

\newcommand{\ad}{\mathrm{ad}}
\newcommand{\Var}{\mathrm{Var}}
\newcommand{\Mot}{\mathrm{Mot}}
\newcommand{\Vect}{\mathrm{Vect}}
\newcommand{\Fil}{\mathrm{Fil}}
\newcommand{\dR}{\mathrm{dR}}

\newcommand{\et}{\text{\'et}}
\newcommand{\eps}{\varepsilon}

\newcommand{\A}{\mathbb{A}}
\newcommand{\F}{\mathbb{F}}
\newcommand{\Q}{\mathbb{Q}}
\newcommand{\Z}{\mathbb{Z}}
\newcommand{\R}{\mathbb{R}}
\newcommand{\C}{\mathbb{C}}
\newcommand{\G}{\mathbb{G}}
\newcommand{\Ss}{\mathbb{S}}
\newcommand{\PP}{\mathbb{P}}
\renewcommand{\L}{\mathbb{L}}

\newcommand{\V}{\mathbf{V}}
\newcommand{\Hh}{\mathbf{H}}
\newcommand{\Pp}{\mathbf{P}}
\newcommand{\LL}{\mathbf{L}}
\newcommand{\PI}{\boldsymbol{\pi}}
\newcommand{\GA}{\boldsymbol{\gamma}}
\newcommand{\XI}{\boldsymbol{\xi}}
\newcommand{\BE}{\boldsymbol{\beta}}
\newcommand{\ET}{\boldsymbol{\eta}}

\newcommand{\OO}{\mathcal{O}}
\newcommand{\Aa}{\mathcal{A}}
\newcommand{\HH}{\mathcal{H}}
\newcommand{\Sh}{\mathcal{S}}
\newcommand{\X}{\mathcal{X}}
\newcommand{\K}{\mathcal{K}}

\newcommand{\h}{\mathfrak{h}}

\title{Galois descent for motives: the K3 case}
\author{Angus McAndrew}
\date{Last updated: \today}

\begin{document}

\maketitle

\begin{abstract}
A theorem of Grothendieck tells us that if the Galois action on the Tate module of an abelian variety factors through a smaller field, then the abelian variety, up to isogeny and finite extension of the base, is itself defined over the smaller field.
Inspired by this, we give a Galois descent datum for a motive $H$ over a field by asking that the Galois action on an $\ell$-adic realisation factor through a smaller field.
We conjecture that this descent datum is effective, that is if a motive $H$ satisfies the above criterion, then it must itself descend to the smaller field.

We prove this conjecture for K3 surfaces, under some hypotheses.
The proof is based on Madapusi-Pera's extension of the Kuga-Satake construction to arbitrary characteristic.
\end{abstract}

\section{Introduction}

The study of descent phenomena can be traced back to some of the earliest mathematicians.
A common proof of the irrationality of $\sqrt{2}$ famously relies on Fermat's method of infinite descent.
Weil, in his seminal work \cite{weil56}, gives a criterion for a variety $X$ over a field $K$ to descend to some subfield $k$.
That is, for there to exist some variety $X_0/k$ such that $X_0 \otimes_k K \cong X$.
This criterion can be phrased in terms of some data associated to the Galois group $\Gal(K/k)$.
We refer to such a criterion as a \emph{Galois descent datum}.
Furthermore, since it is shown that this criterion does indeed imply descent, we say that this datum is \emph{effective}.

Another important example of such a descent datum is the following for abelian varieties.
\begin{dfn}
Let $K/k$ be a regular field extension, that is $K/k$ is finitely generated and separable, and $k$ is algebraically closed in $K$.
Let $A/K$ be an abelian variety.
For $\ell \neq \ch(K)$, let $\rho: \Gal_K \to \GL(V_\ell)$ denote the Galois representation on the $\ell$-adic Tate module.
We say $A$ satisfies \emph{$K/k$-Galois-$\ell$-descent} if $\rho(\Gal(K^s/k^sK))$ is trivial.
\end{dfn}
Note that the criterion only references the separable closure $k^s$, hence at best we can draw conclusions about descent up to a finite extension of $k$.
The following theoreom is due to Grothendieck.
\begin{thm}[\cite{oort73}]\label{prop:grothoort2}
Let $A/K$ satisfy $K/k$-Galois-$\ell$-descent.
Then there exists $A_0/k^s$ such that there exists an isogeny $A_0 \otimes_{k^s} K^s \to A \otimes_K K^s$.
\end{thm}
That is, the $K/k$-Galois-$\ell$-descent criterion for abelian varieties is effective, up to isogeny.
In fact the theorem can be improved to give an isomorphism, rather than isogeny, in the following cases:
\begin{itemize}
\item $A$ is an elliptic curve.
\item $K$ is characteristic 0.
\item $A$ is ordinary.
\end{itemize}
One may wish to generalise this to an arbitrary variety $X$ over $K$.
Rather than the Tate module, we have the $\ell$-adic \'etale cohomology, $H^\ast(X_{\bar K},\Q_\ell)$, which carries a Galois action.
One may then ask if $X$ descends to a subfield $k \subseteq K$ when this Galois representation satisfies an appopriate condition.
However, for general varieties there are generally not enough algebraic maps to be able to draw a conclusion about descent by looking at the cohomology alone.
Instead, when our cohomological hypothesis is satisfied, we might look for the variety to admit an algebraic correspondence with one that is defined over a smaller field.
Since we ask for a condition on the cohomology, it is reasonable to ask for correspondences which induce isomorphisms in $\ell$-adic cohomology.

All this suggests that the natural generalisation is to a category of \emph{motives} over a field $K$.
By motive, we mean an object in any $\Q$-linear category through which cohomology theories on smooth projective varieties factor.
In particular, due to the cohomological criterion, we consider the category $\Mot_{K,\ell}$ of motives \emph{modulo $\ell$-adic equivalence}.
We will motivate and define this category in Section \ref{sec:motives}.

Given a motive $H/K$, we denote $H_\ell$ the $\ell$-adic realisation.
This $\Q_\ell$ vector space is equipped with a Galois representation $\rho: \Gal_K \to \Aut(H_\ell)$.
Thus we can extend the above definition for abelian varieties and say a motive $H$ satisfies \emph{$K/k$-Galois-$\ell$-descent} for $K/k$ a regular extension if $\rho(\Gal(K^s/k^sK))$ is trivial.
We are then lead to conjecture that this criterion is effective.
\begin{conj}\label{conj:intro}
Let $H$ in $\Mot_{K,\ell}$ be a motive satisfying $K/k$-Galois-$\ell$-descent.
Then there exists a motive $H_0/k^s$ in $\Mot_{k^s,\ell}$ such that there exists an isomorphism $H_0 \otimes_{k^s} K^s \to H \otimes_K K^s$.
\end{conj}
Note that the subcategory generated by the motives attached to abelian varieties is precisely the isogeny category.
Hence Grothendieck's theorem can be read as stating that Conjecture \ref{conj:intro} holds for abelian varieties.

Our main result is that the conjecture holds for K3 surfaces, with some hypotheses.
\begin{thm}\label{thm:main}
Conjecture \ref{conj:intro} holds for $X/K$ a K3 surface satisfying either
\begin{enumerate}[(1)]
\item $K$ is characteristic 0.
\item $X$ is ordinary in odd characteristic.
\item $X$ is supersingular in positive characteristic.
\end{enumerate}
\end{thm}
The method of proof is quite different in the three cases, each being deduced from a far stronger result of a different flavour.

The proof in the supersingular uses a result of Bragg and Yang to deduce that the motive of the K3 surface in $\Mot_{K,\ell}$ has a particularly simple form in terms of the Lefschetz motive.
Thus, the motive attached to a supersingular K3 surface will in fact always admit a model over $\Z$, without even needing descent data.

In cases (1) and (2) we will in fact show that the K3 surfaces in question themselves descend.
\begin{thm}\label{thm:K3descent}
Let $K/k$ be a regular extension and $X/K$ a K3 surface.
Let $\ell \neq \ch K$ and $\rho: G_K \to \Aut(H^2(X_{K^s}, \Q_\ell))$ be the Galois representation on the $\ell$-adic cohomology.
Assume $\rho(K^s/k^sK) = 1$ and either (1) $K$ is charactersitic $0$, or (2) $X$ is ordinary in odd characteristic.
Then there exists $X_0/k^s$ and an isomorphism $\varphi: X_0 \otimes_{k^s} K^s \xrightarrow{\sim} X \otimes_K K^s$.
\end{thm}
This uses the Kuga-Satake construction, which associates an abelian variety $A = KS(X)$ to $X$ with certain compatibilities.
The essential strategy is to prove that the $K/k$-Galois-$\ell$-descent passes from $X$ to $A$, and the existence of $A_0$ given by Proposition \ref{prop:grothoort2} implies the existence of $X_0$ as we seek.

The remaining cases of interest, which are not covered in this paper, would be those of K3 surfaces of general finite height.

In Section \ref{sec:motives} we give some background on the theory of motives and Grothendieck's theorem.
In Section \ref{sec:k3} we review the theory of K3 surfaces and the classical Kuga-Satake construction.
We cover the proof of the supersingular case here.
In Section \ref{sec:shimura} we cover the necessary inputs from Madapusi-Pera's work on Shimura varieties to extend the Kuga-Satake construction to arbitrary characterstic.
Section \ref{sec:thm} contains the proof of the remaining cases of the main theorem.

We are thankful to Keerthi Madapusi-Pera for clarifying his work and many helpful suggestions.
We are indebted to Ziquan Yang for reading an early version of this manuscript, offering several improvements and correcting many errors.
We are ever thankful to Jared Weinstein for the insightful questions and guidance.

\section{Abelian varieties and motives}\label{sec:motives}

\subsection{Grothendieck's theorem on CM abelian varieties}

The story of Grothendieck's theorem in fact begins with elliptic curves with complex multiplication.
Recall that we say an elliptic curve $E$ over a field $k$ has complex multiplication, CM, if $\End(E) \neq \Z$.
It is classical that if a complex elliptic curve $E/\C$ has CM, then there exists a number field $k$ and $E_0/k$ with an isomorphism $E_0 \otimes_k \C \cong E$.
See \cite{silverman94}, Theorem 4.3(a).

A similar theorem is true in positive characteristic.
That is, given an elliptic curve $E/K$ where $K$ is a field of characteristic $p$, there exists a finite field $\F_q$, $E_0/\F_q$, and a finite extension $[K':K]<\infty$ such that $E \otimes_K K'\F_q \cong E_0 \otimes_{\F_q} K'\F_q$.
This is sometimes phrased as saying CM elliptic surfaces in positive characteristic are \emph{isotrivial}.
See \cite{mumford70}, page 217.

In fact this can be extended to the case of abelian varieties.
First recall that a simple abelian variety $A$ has CM if $\End^0(A)$ is a CM field of degree $2 \dim A$.
Then an abelian variety $A$ has CM if it is isogenous to a product of simple CM abelian varieties.
The descent theorem on CM was extended to abelian varieties by Grothendieck, as espoused in \cite{oort73}.
The theorem will crucially use the Galois action on the Tate module.
Hence we must first make a certain definition about our field extensions to make sure the Galois theory behaves well.
\begin{dfn}
We say a finitely generated field extension $K/k$ is \emph{separable} if there exists a trascendence basis $B$ of $K$ such that $K/k(B)$ is a separable algebraic extension.
We say a finitely generated field extension $K/k$ is \emph{regular} if $K/k$ is separable and $k$ is algebraically closed in $K$.
\end{dfn}
We may now state Grothendieck's theorem.
\begin{thm}[\cite{oort73}, Thm 1.1]
Let $A/K$ be an abelian variety with CM.
Let $k$ be the prime field of $K$ such that $K/k$ is regular.
Then there exists an abelian variety $A_0/k^s$ with an isogeny $A_0 \otimes_{k^s} K^s \to A \otimes_K K^s$.
\end{thm}
The proof of this is built from two theorems.
The first is to relate the CM condition to a property of the Galois action on the Tate module.
\begin{thm}[\cite{oort73}, `First Step']
Let $A/K$ be an abelian variety with CM, with $K$ a field regular over its prime field $k$.
For $\ell \neq \ch(K)$, let $\rho: G_K \to \GL(V_\ell)$ denote the Galois representation on the $\ell$-adic Tate module.
Then $\rho(\Gal(K^s/k^sK))$ is finite.
\end{thm}
The strategy of proof is to induct over transcendence degree, epxressing $K/k$ as a sequence of extensions of transcendence degree one.
\begin{rmk}
Note that the conclusion can always be improved to $\rho(\Gal(K^s/k^sK))$ being trivial if one passes to some $[K':K] < \infty$.
\end{rmk}
The next step in Grothendieck's argument is the following theorem.
\begin{thm}[\cite{oort73}, `Last Step']\label{thm:groth}
Let $A/K$ be an abelian variety with $K/k$ a regular extension.
For $\ell \neq \ch(K)$, let $\rho: G_K \to \GL(V_\ell)$ denote the Galois representation on the $\ell$-adic Tate module and assume $\rho(\Gal(K^s/k^sK)) = 1$.
Then there exists $A_0/k$ such that there exists an isogeny $A_0 \otimes_k K \to A$.
\end{thm}
\begin{proof}
The proof relies on the theory of the $K/k$-trace of $A$, which is the largest isogeny factor of $A$ that descends to $k$.
Precisely, we define the $K/k$-trace of $A$ to be a pair $(A_0/k, t)$ with $A_0$ an abelian variety and $t: A_0 \otimes_k K \to A$ a $K$-homomorphism with finite kernel which is final for all such pairs.
Thus for the theorem it suffices to show that $t$ is an isogeny.

The conditions on $K/k$ ensure we are in the situation of the Lang-N\'eron theorem, an analogue of the Mordell-Weil theorem.
Precisely, we have that $A(k^sK)/t(A_0(k^s))$ is a finitely generated abelian group.
See \cite{conrad06} for full details.
Then we have that $t(A_0(k^s))[\ell^n]$ is finite index in $A(k^sK)[\ell^n]$, of index bounded independent of $n$.
Further, since $\rho(\Gal(K^s/k^sK)) = 1$ we have that $A(k^sK)[\ell^n] = A(K^s)[\ell^n]$.
Taking the direct limit over $n$ we see $A(K^s)[\ell^\infty]$ and $A_0(k^s)[\ell^\infty]$ have the same rank as $\Q_\ell / \Z_\ell$-modules.
Hence $A$ and $A_0$ have the same dimension and thus $t$ is an isogeny, as required.
\end{proof}
\begin{rmk}\label{rmk:groth}
In general the above theorem only returns an isogeny $A_0 \otimes K \to A$.
However, in the case of elliptic curves one can always find an isomorphism $E_0 \otimes K \xrightarrow{\sim} E$.
Further, as explained in \cite{oort73} 1.3, the isogeny is an isomorphism if $\ch K = 0$ or if $A$ has $p$-rank $\geq g-1$.
In particular this is satisfied if $A$ is ordinary.
\end{rmk}

\subsection{Motives and Galois descent}

We are interested in exploring how theorem \ref{thm:groth} can be generalised.
To extend it to general smooth projective varieties $X$, the replacement for the $\ell$-adic Tate module is the $\ell$-adic \'etale cohomology $H^i(\overline X, \Q_\ell)$.
However, even for a fixed $X$ there are distinct Galois representations for each $i$, and even for a fixed $i$ we have many interesting subrepresentations of $H^i(\overline X, \Q_\ell)$.
The natural category that captures these structures is the category of \emph{motives}.
This is a $\Q$-linear category through which all cohomology theories factor.
If we let $\Var_K$ be the category of smooth, projective varieties over $K$ and denote $\Mot_K$ a category of motives, then it fits in the following diagram, where one may also include $H^\ast_B$ and $H^\ast_\mathrm{cris}$, depending on $\ch K$.
\[
\xymatrix{
&&&& \Vect_{\Q_\ell} \\
\Var_K \ar[rr] \ar[rrrru]^{H^i_{\et}(\cdot,\Q_\ell)} \ar[rrrrd]_{H^i_{\text{dR}}(\cdot)} && \Mot_K \ar[rru] \ar[rrd] && \\
&&&& \Vect_K
}
\]
For concreteness we will describe below Grothendieck's construction with respect to an equivalence relation on cycles, $\sim$.
The key examples would be rational, algebraic, numerical, and homological equivalence.
See \cite{milne13} for full details.

For $\sim$ an equivalence relation on cycles, we have that category of motives, $\Mot_{K,\sim}$, is defined by
\begin{center}
\begin{tabular}{rl}
\textbf{Objects:} & $\h(X,e,m)$ \\
\textbf{Morphisms:} & $\Hom(\h(X,e,m),\h(Y,f,n)) = f \circ C^{\dim X + n - m}_\sim(X \times Y)_\Q \circ e$
\end{tabular}
\end{center}
where:
\begin{itemize}
\item $X, Y$ are smooth, projective varieties.
\item $m, n$ are integers.
\item $C^{\dim X + n - m}_\sim(X \times Y)_\Q$ is the group of codimension $\dim X + n - m$ cycles on $X \times Y$ up to the equivalence relation $\sim$, tensored with $\Q$ over $\Z$.
\item $e, f$ are idempotents in $C^{\dim X}_\sim(X \times X)_\Q, C^{\dim Y}_\sim(Y \times Y)_\Q$, respectively.
\end{itemize}
One example of particular relevance for our purposes will be $\Mot_{K,\ell}$ of motives up to $\ell$-adic homological equivalence.

The idea is to use such cycles as homomorphisms to capture all maps on cohomology induced by correspondences.
The role of the idempotents is to single out parts of cohomology beyond the full $H^\ast(X)$.
The addition of the integer $m$ allows for duals to exist.

Given such a motive, it admits \emph{realisations} in the various cohomology theories, by $\h(X,e,m) \mapsto e^\ast H^\ast(X) \otimes H^2(\PP^1)^{\otimes m}$.
\begin{ex}
If $\Delta_X \subseteq X \times X$ is the diagonal, then we define $\h(X) = \h(X, \Delta_X, 0)$.
Under any realisation this gives the full cohomology group $H^\ast(X)$.

If $X$ is defined over a finite field $\F_q$ and $\ell \nmid q$, let $g_i$ be the characteristic polynomial of Frobenius on $H^i(\overline X, \Q_\ell)$.
Find a polynomial $G$ that is $1 \pmod{g_i}$ and $0 \pmod{g_j}$ for $j \neq i$, this is possible due to the Riemann Hypothesis.
Let $\Gamma_F$ denote the graph of Frobenius on $X \times X$.
Then $\h(X, G(\Gamma_F), 0)$ under any realisation gives the cohomology group $H^i(X)$.
\end{ex}
Over a general field $K$, the existence of a cycle which cuts out precisely $H^i(X)$ is an open conjecture, known as the K{\''u}nneth type Standard Conjecture.
We call such a cycle the $i$th K{\''u}nneth projector, denoted $\pi^i_X$ and write $\h^i(X) = (X,\pi^i_X,0)$.
\begin{ex}
We have a decomposition $\h(\PP^1) = \h^0(\PP^1) \oplus \h^2(\PP^1)$.
We write $\L = \h^2(\PP^1)$, this is known as the \emph{Lefschetz motive}.
\end{ex}

We now are in position to discuss Galois descent.
For an abelian variety $A$, we have that $V_\ell A \cong H^1_\et(\overline A, \Q_\ell)^\vee$.
Hence, for a general motive, we are lead to the following.
\begin{dfn}
Let $K/k$ be a regular field extension and $H$ be a motive in $\Mot_{K,\ell}$.
Let $H_\ell$ be the $\ell$-adic realisation, equipped with Galois representation $\rho: \Gal_K \to \Aut(H_\ell)$.
We say $H/K$ satisfies \emph{$K/k$-Galois-$\ell$-descent} if $\rho(\Gal(K^s/k^sK))$ is trivial.
\end{dfn}
Given this, we can now state the main conjecture, i.e. that this descent criterion is \emph{effective}.
\begin{conj}\label{conj:motive}
Let $H$ be a motive in $\Mot_{K,\ell}$ satisfying $K/k$-Galois-$\ell$-descent.
Then there exists a motive $H_0$ be a motive in $\Mot_{k^s,\ell}$ such that there exists an isomorphism $H_0 \otimes_{k^s} K^s \to H \otimes_K K^s$.
\end{conj}
\begin{rmk}
Since the category of motives is $\Q$-linear and the morphisms are induced by correspondences, this can be viewed as analogous to saying that this criterion for a variety $Y/K$ predicts a variety $Y_0/k^s$ and an algebraic correspondence $Y_0 \otimes_{k^s} K^s \to Y \otimes_K K^s$ that induces an isomorphism on cohomology.
\end{rmk}
\begin{prop}\label{prop:conjavs}
Conjecture \ref{conj:motive} holds for the sub-tensor-category of $\Mot_{K,\ell}$ generated by $\h^1(A)$ for all abelian varieties $A$.
\end{prop}
\begin{proof}
By theorem \ref{thm:groth}, if $\h^1(A)$ satisfies $K/k$-Galois-$\ell$-descent, then $A$ is isogenous to some $A_0/k^s$.
Indeed, such maps induce isomorphisms on cohomology.
Thus in particular $\h^1(A_0) \otimes_{k^s} K^s \cong \h^1(A) \otimes_K K^s$, as required.
\end{proof}

\section{K3 surfaces}\label{sec:k3}

\subsection{K3 surfaces}

Elliptic curves are the only smooth projective curves which admit a group law.
Hence it is natural to generalise theorems for elliptic curves to the case of abelian varieties.
Another feature uniquely specifying elliptic curves among smooth projective curves is that they have trivial canonical bundle.

Generalising this to higher dimensions give the notion of a \emph{Calabi-Yau} variety.
Examples of these are given by abelian varieties in each dimension, however these do not constitute all examples.
In the case of surfaces we may ask for a Calabi-Yau surface, but further ask that it is simply connected.
\begin{dfn}\label{def:k3}
A K3 surface $X/K$ is a smooth projective surface over $K$ such that $\omega_X \cong \OO_X$ and $H^1(X,\OO_X) = 0$.
\end{dfn}
Perhaps the most striking feature of a K3 surface is its cohomology.

We begin in characteristic 0, where this is captured by the Hodge diamond, given below.
\[
\begin{array}{ccccc}
&& 1 && \\
& 0 && 0 & \\
1 && 20 && 1 \\
& 0 && 0 & \\
&& 1 &&
\end{array}
\]
In particular, $H^2(X_\C,\Z)$ equipped with the intersection pairing has the structure of a quadratic lattice, isomorphic to the so-called K3 lattice $L = U^{\oplus 3} \oplus E_8^{\oplus 2}$, where $U$ is the hyperbolic lattice.
A \emph{polarisation} of $X$ is a class of ample line bundle $\xi$ on $X$ of degree $d$.
A polarisation determines a primitive cohomology subgroup $PH^2_\xi(X,\Z) = \langle \xi \rangle^\bot \subseteq H^2(X_\C,\Z)$.
We set $L_d = \langle e-df \rangle^\bot \subseteq L$ where $d$ is the degree of $\xi$ and $e,f$ are basis elements of one copy of $U$ satisfying $e^2 = f^2 = 0$ and $\langle e,f \rangle = 1$.
Then there further exists an isomorphism of quadratic lattices $PH^2_\xi(X,\Z) \cong L_d$.
These cohomologies are weight 2 Hodge structures, which viewed as representations of $\Ss$ have image in $\SO(L)$ and $\SO(L_d)$ respectively.

In arbitrary characteristic we instead consider $\ell$-adic cohomology.
Here, as a module with quadratic form, there exist isomorphisms $H^2(X,\Z_\ell) \cong L \otimes \Z_\ell$ and $PH^2_\xi(X,\Z_\ell) \cong L_d \otimes \Z_\ell$.

As motives, we have $\h(X) = \h^0(X) \oplus \h^2(X) \oplus \h^4(X) = \L^{\otimes 0} \oplus \h^2(X) \oplus \L^{\otimes 2}$.

In \cite{artinmazur}, Artin and Mazur introduce formal groups associated to the cohomology of algebraic varieties.
In particular, if $\ch(K) = p > 0$, associated to a K3 surface $X$ we have the \emph{formal Brauer group} $\widehat{\Br}_X$, which is a smooth one-dimensional formal group.
As with any such one-dimensional formal groups $\widehat{\Br}_X$ in positive characteristic, we have the \emph{height} invariant.
\begin{dfn}
Let $K$ be a field such that $\ch(K) = p > 0$.
Let $X/K$ be a K3 surface.
The \emph{height} of $X$ is the height of the formal Brauer group $\widehat{\Br}_X$ of $X$.
We say $X$ is \emph{ordinary} if it has height 1.
We say $X$ is \emph{supersingular} if it is of infinite height.
\end{dfn}
There are alternative characterisations of ordinary and supersingular K3 surfaces.
In the case of a finite field, the category of ordinary K3 surfaces has been given a classification in terms of certain linear algebraic data in \cite{taelman20}.

\subsection{The supersingular case}\label{sec:super}

An equivalent characterisation of supersingular K3 surfaces is that they are precisely the surfaces in positive characteristic with maximal Picard rank $\rho(X) = 22$.
(Note that in characteristic zero the Picard rank is bounded by $h^{1,1} = 20$.)
One might therefore expect that the corresponding motive $h(X)$ would take a particularly simple form.
For instance, if one considers the $\ell$-adic realisation, we see that we have an isomorphism $NS(X) \otimes \Q_\ell \xrightarrow{\sim} H^2_\et(X,\Q_\ell(1))$.

We now turn to the results of \cite{braggyang21}.
In particular, Theorem 6.11 gives that for $p$ odd we have an isomorphism between the motives $\h^2(X)$ and $\h^2(Y)$ for any two supersingular K3 surfaces $X,\ Y$.
Following Remark 6.12, we start with arbitrary supersingular $X$ and apply this when $Y$ is a supersingular Kummer surface.
Then \cite{shioda77}[Theorem 1.1] gives that supersingular Kummer surfaces $Y$ are unirational, from which one finds that $\h^2(Y) = \L^{\oplus 22}$.

\begin{proof}[Case 3 of Theorem \ref{thm:main}]
Let $K$ be a field of characteristic $p > 0$ and $X/K$ be a supersingular K3 surface.
(We will not need to make any Galois descent assumption.)
As motives, we have $\h(X) = \L^{\otimes 0} \oplus \L^{\oplus 22} \oplus \L^{\otimes 2}$.
However, the Lefschetz motive is $\L = \h^2(\PP^1)$, which admits a model over $\Z$.
Hence the same is true for the motive $\h(X)$ and it can be descended to any field $k \subseteq K$.
\end{proof}

\subsection{The Clifford algebra and $\GSpin$}

Let $L$ be a quadratic lattice with quadratic form $q$ over a commutative ring $R$ in which $2$ is not a zero divisor.
We have the Clifford algebra
\[
\Cl(L) = (\oplus_{n\geq0} L^{\otimes n}) / (v \otimes v - q(v)),
\]
which comes with a natural $\Z/2\Z$-grading.
We define the group $\GSpin(L)$ as
\[
\GSpin(L) = \{ v \in \Cl(L)^\times ~|~ vLv^{-1} = L \}.
\]
(Warning: In other references this is often referred to as $\mathrm{CSpin}(L)$.)
We will use the same notation $\GSpin(L)$ for an algebraic group over $R$ such that $\GSpin(L)(R') = \GSpin(L \otimes R')$.
We get an orthogonal representation of $\GSpin(L)$ on $L$ via conjugation, on which one can see the scalars act trivially.
This gives the following short exact sequence, expressing $\GSpin(L)$ as a central extension of $\SO(L)$.
\[
1 \to \G_m \to \GSpin(L) \xrightarrow{\ad} \SO(L) \to 1
\]
There is an anti-involution $c \mapsto c^\ast$ on $\Cl(L)$ induced by reversing the order of simple tensors.
We have the \emph{spinor norm}
\begin{align*}
\nu : \GSpin(L) &\to \G_m \\
x &\mapsto x^\ast x.
\end{align*}
Note that restricting to $\ker \ad \cong \G_m$ gives $\nu(x) = x^2$.

Further, on $\Cl(L)$ we have an $R$-linear reduced trace map $\Trd: \Cl(L) \to R$ such that $(x,y) \mapsto \Trd(xy)$ is a non-degenerate symmetric bilnear form on $\Cl(L)$.
\begin{lem}\label{lem:spinsymp}
For $\delta \in \Cl(L)^\times$ such that $\delta^\ast = -\delta$, the form $\psi_\delta(x,y) = \Trd(x\delta y^\ast)$ defines an $R$-valued symplectic form on $\Cl(L)$.
For $g \in \GSpin(L)$ we have
\begin{align*}
\psi_\delta(gx,gy) = \nu(g) \psi_\delta(x,y).
\end{align*}
\end{lem}
\begin{proof}
To see that $\psi_\delta$ is alternating and hence symplectic we compute
\[
\Trd(x\delta x^\ast) = \Trd((x\delta x^\ast)^\ast) = \Trd(x \delta^\ast x^\ast) = - \Trd(x\delta x^\ast),
\]
noting that $\Trd(x) = \Trd(x^\ast)$.

For $g \in \GSpin(L)$, we compute
\begin{align*}
\psi_\delta(gx,gy) = \Trd((gx)\delta(gy)^\ast) &= \Trd(g(x\delta y^\ast)g^\ast) \\
&= \Trd(g(x\delta y^\ast)\nu(g)g^{-1}) = \nu(g) \Trd(x\delta y^\ast) = \nu(g)\psi_\delta(x,y),
\end{align*}
noting that $g^\ast = \nu(g)g^{-1}$ and $\Trd$ is conjugation-invariant.
\end{proof}
\begin{rmk}
\begin{itemize}\label{rmk:GSpinGSp}
\item The above gives a map $\GSpin(L) \hookrightarrow \GSp(\Cl(L),\psi_\delta)$ under which the spinor norm corresponds to the similitude character.
\item There is an isomorphism $\Cl(L) \cong \End_{\Cl(L)}(\Cl(L))$ of $\GSpin(L)$-representations, where it acts on the right by right multiplication and on the left by conjugation.
The inclusion $L \subseteq \Cl(L)$ is $\GSpin(L)$-equivariant only with respect to the conjugation action.
\end{itemize}
\end{rmk}

\subsection{Hodge structures and Kuga-Satake}

Now consider $L$ a K3 Hodge structure.
That is, a Hodge structure $L$ of type $\{(2,0), (1,1), (0,2)\}$ such that $\dim L^{2,0} = \dim L^{0,2} = 1$.
\begin{ex}
If $X$ is a complex K3 surface or complex abelian surface then $H^2(X,\Z)$ is a K3 Hodge structure.
For any sub-Hodge structure $V \subseteq L^{1,1}$ then $V^\bot$ is a K3 Hodge structure, e.g. $PH^2(X,\Z)$, or the transcendental lattice $T(X)$.
\end{ex}
Under these circumstances, $\Cl(L_\R)$ can be equipped with a complex structure, or equivalently a weight 1 Hodge structure, i.e. of type $\{(1,0), (0,1)\}$.
If we denote the Hodge structure on $L$ as $\rho : \Ss \to \SO(L_\R)$, where $\Ss$ is the Deligne torus, then this weight 1 structure can be viewed as exhibiting a lift as in the following diagram.
\[
\xymatrix{
& \GSpin(L_\R) \ar[r]^{\sp} \ar[d]^\ad & \GSp(\Cl(L_\R)) \\
\Ss \ar_-{\rho(1)}[r] \ar^-{\tilde\rho}[ur] & \SO(L_\R(1))
}
\]
where we take the unique lift such that $\Nm \circ \tilde\rho$ is the norm map $\Ss \to \G_{m,\R}$.
The composition $\sp \circ \tilde\rho$ endows $\Cl(L)$ with a weight 1 Hodge structure (see \cite{deligne72}) and allows us to form a complex torus $\KS(L) = \Cl(L_\R)/\Cl(L)$.

If further the Hodge structure $L$ comes with a polarisation, it is possible to define a polarisation on $KS(L)$.
\begin{dfn}
Given a complex K3 surface $X$ the above construction applied to $L = PH^2(X,\Z)$ gives the \emph{Kuga-Satake abelian variety} $A = \KS(X)$.
\end{dfn}

\section{Shimura Varieties}\label{sec:shimura}

We can descend the above construction to algebraic K3 surfaces $X$ defined over other fields $k$ of characteristic 0, by simply applying it to $X_\C$.
The result of this would be some complex abelian variety $\KS(X_\C)$.
Work of Deligne in \cite{deligne72} shows that in fact there exists an abelian variety $A/\bar k$ such that $A_\C = \KS(X_\C)$.

Madapusi-Pera in \cite{pera15} extends this construction to fields of arbitrary characteristic.
Similar constructions can be found in \cite{rizov10} and \cite{maulik14}.
An excellent summary is available in \cite{yang20}.
This is done by packaging the above as certain morphisms of moduli spaces and Shimura varieties and extending these to the canonical integral models.
The relevant diagram here is as follows
\[
\xymatrix{
& \Sh(\GSpin(L_d)) \ar[r] \ar[d] & \Sh(\GSp(\Cl(L_d),\psi_\delta)) \\
\tilde M_{2d,\gamma} \ar[r] & \Sh(\SO(L_d))
}
\]
where:
\begin{itemize}
\item $\tilde M$ is a moduli space of ``$\gamma$-oriented'' K3 surfaces (quasi-polarised of degree $2d$).
\item $\Sh(\cdot)$ is the integral model of a Shimura variety $\SH(\cdot)$.
\item All morphisms are finite \'etale over their image.
\end{itemize}
The goal of this section is to how this diagram is constructed, and how it allows us to associate a Kuga-Satake abelian variety to a K3 surface.

We fix $p \in \Z$ a prime.

\subsection{Shimura varieties over number fields}

Recall $L = U^{\oplus 3} \oplus E_8^{\oplus 2}$ is the K3 lattice, equipped with its quadratic form.
For $d \in \Z$ we have $L_d = \langle e-df \rangle^\bot \subseteq L$, where $U = \langle e,f \rangle$.
Let $\Omega$ be the space of oriented negative definite planes in $L_d \otimes \R$.
\begin{lem}[\cite{pera16}, 3.1]
The pairs $(\GSpin(L_d \otimes \Q),\Omega)$ and $(\SO(L_d \otimes \Q), \Omega)$ are Shimura data with reflex field $\Q$.
\end{lem}
We have the Clifford algebra $\Cl(L_d)$.
As in Lemma \ref{lem:spinsymp}, for each $\delta \in \Cl(L_d)^\times$ there exists a symplectic form $\psi_\delta$ on $\Cl(L_d)$ and hence an embedding $\GSpin(L_d) \hookrightarrow \GSp(\Cl(L_d), \psi_\delta)$.
Let $\HH_\delta$ be the union Siegel half-spaces attached to $(\Cl(L_d), \psi_\delta)$.
Then for each $\delta$ we get the classical Siegel Shimura datum $(\GSp(\Cl(L_d \otimes \Q), \psi_\delta), \HH_\delta)$ with reflex field $\Q$.
\begin{lem}[\cite{pera16}, Lemma 3.6]
There exists a choice of $\delta$ such that the embedding $\GSpin(L_d) \hookrightarrow \GSp(\Cl(L_d), \psi_\delta)$ induces an embedding of Shimura data $(\GSpin(L_d \otimes \Q),\Omega) \hookrightarrow (\GSp(\Cl(L_d \otimes \Q), \psi_\delta), \HH_\delta)$.
\end{lem}
The adjoint representation $\GSpin(L_d) \to \SO(L_d)$ also induces a morphism of Shimura varieties, giving the diagram
\[
\xymatrix{
\SH(\GSpin(L_d)) \ar[r] \ar[d] & \SH(\GSp(\Cl(L_d),\psi_\delta)) \\
\SH(\SO(L_d))
}
\]

\subsection{Bundles and abelian varieties}\label{sec:bundles}

To describe how this relates to K3 surfaces an abelian varieties, we now describe certain bundles on these Shimura varieties.
These will be set up in such a way as to be amenable to integrality over $\Z_{(p)}$, discussed in the following section.

For any Shimura datum $(G,X)$ we have an isomorphism of complex manifolds
\[
\SH_K(G)(\C) \cong G(\Q) \setminus (X \times G(\A_f)) / K.
\]
This allows us to construct a functor taking an algebraic representation $V$ of $G$ to a bundle on $\SH_K(G)(\C)$.
First by taking the constant local system $V \times X \times G(\A_f) / K$, then quotienting by the action of $G(\Q)$ gives a local system $\V_B$ on $\SH_K(G)(\C)$.
We can further equip the bundle $(V \otimes \C) \times X$ with a filtration such that at each point $h \in X$ the filtration is induced by the homomorphism $h: \mathbb{S} \to G \otimes \R$ at that point.
This descends to $\SH_K(G)(\C)$, giving an \emph{algebraic variation of Hodge structures}, denoted $(\V_B, \Fil^\bullet \V_{\dR,\C})$ (see \cite{pera16}, 3.3).

Applying the above construction to the representation of $\GSpin(L_d)$ on $\Cl(L_d)$ by right multiplication gives an algebraic variation of Hodge structures $(\Hh_B, \Fil^\bullet \Hh_{\dR,\C})$ on $\SH_K(\GSpin(L_d))(\C)$.
This is further equipped with a $\Z/2\Z$-grading, a right $\Cl(L_d)$-action, and tensors
\begin{align*}
\PI_B &\in H^0\left(\SH_K(\GSpin(L_d))(\C), (\Hh_B^{\otimes 2} \otimes (\Hh_B^\vee)^{\otimes 2}) \otimes \Q\right), \\
\GA_B &\in H^0\left(\SH_K(\GSpin(L_d))(\C), \det(\LL_B)\right),
\end{align*}
where $(\LL_B, \Fil^\bullet \LL_{\dR,\C})$ is the variation of Hodge structures arising from the adjoint representation of $\GSpin(L_d)$ into $\SO(L_d)$, identified with $\PI_B((\Hh_B, \Fil^\bullet \Hh_{\dR,\C}))$.

The Shimura variety $\SH_K(\GSp(\Cl(L_d),\psi_\delta))$ admits a natural moduli description.
Write $K = K_p K^p$ for $K_p \subseteq \GSp(\Q_p)$ and $K^p \subseteq \GSp(\A_f^p)$.
We have a functor which takes a scheme $T$ to a polarised abelian variety up to prime-to-$p$ isogeny with level structure.
Precisely, let $A$ be a prime-to-$p$ isogeny class of abelian varieties and $\lambda: A \to A^\vee$ a $\Z_{(p)}^\times$-equivalence class of a polarisation.
Let $\underline{\mathrm{Isom}}\left(\underline{\Cl(L_d) \otimes \A_f^p}, R^1 f_{\et,\ast} \underline{\A}_f^p\right)$ be the isomorphisms compatible with $\psi_\delta$ and $\lambda$ up to an element of $(\A_f^p)^\times$.
Then the functor gives is a triple $(A, \lambda, \eps)$, where $A$ is a prime-to-$p$ isogeny class of abelian varieties, $\lambda: A \to A^\vee$ is the $\Z_{(p)}^\times$-equivalence class of a polarisation, and $\eps \in H^0\left(T, \underline{\mathrm{Isom}}\left(\underline{\Cl(L_d) \otimes \A_f^p}, R^1 f_{\et,\ast} \underline{\A}_f^p\right)/K^p\right)$ is a $K$-level structure, where $f: A \to T$ is the structure morphism.
This functor is representable over $\Z_{(p)}$ by a scheme $\Sh_K(\GSp(\Cl(L_d)),\psi_\delta)$ whose generic fibre is precisely $\SH_K(\GSp(\Cl(L_d),\psi_\delta))$.

We thus have the tautological bundle $(\Aa, \lambda, \eps)$ over $\SH_K(\GSp(\Cl(L_d)),\psi_\delta)$.
Pulling this back to $\SH_K(\GSpin(L_d))$ gives a bundle $(\Aa^\KS_\Q, \lambda^\KS_\Q, \eps^\KS_\Q)$, where $\Aa^\KS_\Q$ is known as the \emph{Kuga-Satake abelian scheme}.
On $\SH_K(\GSpin(L_d))(\C)$ the algebraic variation of Hodge structures arising from the induced $\GSpin(L_d)$ representation on $H^1(\Aa^\KS_\C)$ is identified with $(\Hh_B, \Fil^\bullet \Hh_{\dR,\C})$.
Thus $\Aa^\KS_\C$ carries a $\Z/2\Z$-grading and $\Cl(L_d)$-action.
By \cite{pera16}, Proposition 3.11, the identification of bundles, along with the $\Z/2\Z$-grading and $\Cl(L_d)$-action, descend to $\Aa^\KS_\Q$ on $\SH_K(\GSpin(L_d))$ over $\Q$.

Conversely, we can construct bundles $\Hh_\bullet$ with tensors $\PI_\bullet, \GA_\bullet$ on $\Sh_K(\GSpin(L_d))$ for $\bullet = \ell, \dR, \mathrm{cris}$ using the cohomology of $\Aa^\KS$.
Similarly, we construct $\LL_\bullet$ on $\Sh_K(\SO(L_d))$ by $\PI_\bullet(\Hh_\bullet)$.

\subsection{Integral models}

By the moduli interpretation above, we have an integral model $\Sh_K(\GSp(\Cl(L_d)),\psi_\delta)$.
From here we will abuse notation and use $K$ to be a compatible choice of open compact subgroup of each of $\GSpin(L_d)(\A_f)$, $\SO(L_d)(\A_f)$ and $\GSp(\Cl(L_d),\psi_\delta)(\A_f)$ simultaneously.
We further choose $K$ sufficiently small so that $\SH(\GSp(\Cl(L_d),\psi_\delta))$ admits a description as a fine moduli scheme.
The definition of integral canonical model can be found in \cite{pera16}, Definition 4.3, relying on a certain extension property, analogous to the N\'eron model of an abelian variety.
We state the form of this definition found in \cite{kisin10} 2.3.7.
\begin{dfn}\label{dfn:ext}
A scheme $X$ over $\Z_{(p)}$ satisfies the extension property if, for any regular, formally smooth $\Z_{(p)}$-scheme $S$, any map $S \otimes \Q \to X$ extends to a map $S \to X$.
\end{dfn}
We now have the main result on integral models for the Shimura varieties in question.
\begin{thm}[\cite{pera16} Theorem 4.4, \cite{kisin10} 2.3.8, \cite{vasiu99} 3.4.14]\label{thm:integral}
Let $\Sh_K(\GSpin(L_d))$ be the Zariski closure of $\SH_K(\GSpin(L_d))$ in $\Sh_K(\GSp(\Cl(L_d),\psi_\delta))$.
Then $\Sh_K(\GSpin(L_d))$ is a smooth integral canonical model for $\SH_K(\GSpin(L_d))$.
Further, the finite Galois cover $\SH_K(\GSpin(L_d)) \to \SH_K(\SO(L_d))$ extends to $\Sh_K(\GSpin(L_d)) \to \Sh_K(\SO(L_d))$,  where $\Sh_K(\SO(L_d))$ is a smooth integral canonical model for $\SH_K(\SO(L_d))$.
\end{thm}
\begin{rmk}
In the above references in fact the integral model is constructed as the normalisation of the Zariski closure.
However, \cite[Theorem 1.1.1]{xu20} shows that the closure is already smooth and thus the normalisation step is unecessary.
\end{rmk}
The polarised abelian scheme discussed in section \ref{sec:bundles} extends to $(\Aa^\KS, \lambda^\KS, \eps^\KS)$ over this model.
Further, the $\Z/2\Z$-grading and $\Cl(L_d)$-action extend to $\Aa^\KS$ by the theory of N\'eron models, see \cite{pera16} 4.5, \cite{faltingschai90} I.2.7.

\subsection{The Period Morphism}

To discuss the moduli of K3 surfaces, we begin by extending Definition \ref{def:k3} to families.
We also introduce some additional structure for our moduli problem to allow us to relate it to Shimura varieties.
\begin{dfn}
A K3 surface $f: X \to S$ over a scheme $S$ is a smooth and proper algebraic space such that the geometric fibres are K3 surfaces.
A polarisation is a section $\xi \in \underline{\Pic}(X/S)(S)$ whose fibre at each geometric point is an ample line bundle.
A quasi-polarisation is a section $\xi \in \underline{\Pic}(X/S)(S)$ whose fibre at each geometric point is a big and nef line bundle.
If $\xi(s)$ is degree $2d$ for all geometric points $s \to S$, we say $\xi$ has degree $2d$.
If $\xi(s)$ is primitive for all geometric points $s \to S$, we say $\xi$ is primitive.
\end{dfn}
We have a functor which sends each $\Z[1/2]$-scheme $S$ to the set of pairs $(f: X \to S, \xi)$ where $f:X \to S$ is a K3 surface and $\xi$ is a primitive quasi-polarisation of deg $2d$.
This is represented by a Deligne-Mumford stack $M_{2d}$ of finite type over $\Z$.
There is an open substack $M_{2g}^\circ$ on which $\xi$ is a polarisation.

We have the universal object $(\X \to M_{2d}, \XI)$ over $M_{2d}$.
We get vector bundles $\Pp^2_\dR, \Pp^2_B, \Pp^2_\ell, \Pp^2_\mathrm{cris}$ over $M_{2d}, M_{2d,\C}, M_{2d,\Z[1/2\ell]}, M_{2d,\F_p}$, respectively, arising from the second primitive cohomology of $\X$.
That is $\Pp^2_\bullet = \langle \cl(\XI) \rangle^\bot$, i.e. the complement of the class of $\XI$ in the second cohomology.
Note that we can also package the $\ell$-adic cohomology into $\Pp^2_{\hat\Z} = \prod_\ell \Pp^2_\ell$.
We have similar constructions for the full second cohomology, denoted $\V_\bullet$.

We may now introduce the notion of a $\gamma$-orientation.
\begin{dfn}
Fix $\gamma \in \det(L_d)$.
Let $\tilde M_{2d,\gamma}$ over $\Z[1/2]$ be the Deligne-Mumford stack such that for a scheme $T$, the points in $\tilde M_{2d}(T)$ are pairs $(T \to M_{2d}, \BE_T)$, where $\BE_T \in H^0(T,\det(\Pp^2_{2,T}))$ such that $\langle \BE_T, \BE_T \rangle$ is the constant section $\langle \gamma,\gamma \rangle$.
We call $\tilde M_{2d,\gamma}$ the space of \emph{$\gamma$-oriented quasi-polarised K3 surfaces of deg $2d$}.
\end{dfn}
The map that forgets $\BE_T$ gives a degree 2 \'etale cover $\tilde M_{2d,\gamma} \to M_{2d}$.

Let $K$ be an \emph{admissible} compact open subgroup of $\SO(L_d)(\A_f)$, i.e. one that stabilises $L \otimes \hat\Z$.
Analogously to the case of abelian varieties, we define a $K$-level structure on $\tilde M_{2d}$ over a base $T$ by first considering the set
\[
I  = \{ \text{isometries } \eta : L \otimes \hat\Z \xrightarrow{\sim} \V^2_{\hat\Z} \text{ such that } \eta(e-df) = \mathrm{ch}_{\hat\Z}(\XI), (\wedge^m \eta)(\delta \otimes 1) = (\BE_{\hat\Z}) \}
\]
on which $K$ acts.
Then a section $\ET \in H^0(T, I/K)$ is a $K$-level structure over $T$.

As in \cite[Proposition 3.10]{milne94}, \cite[Proposition 3.3]{pera15}, the Shimura variety $\SH(\SO(L_d))_\C$ admits a modul interpretation in terms of certain Hodge structures (\emph{families of $\Z$-motives}, in the language of the latter).
In particular, it admits a universal property which induces a map $\iota_\C: \tilde M_{2d,\gamma,\C} \to \SH(\SO(L_d))_\C$, see \cite[Proposition 4.2]{pera15}.
Then \cite[Corollary 4.4]{pera15} allows us to extend $\iota_\C$ first to $\iota_\Q$.
Similar results may also be found in \cite[Theorem 3.16]{rizov05} and \cite[Proposition 5.7]{maulik14}.
Recalling that the integral model $\Sh(\SO(L_d))$, by Theorem \ref{thm:integral}, satisfies the extension property (see Definition \ref{dfn:ext}), this results in the \emph{period morphism}
\[
\iota: \tilde M_{2d,\gamma,K} \longrightarrow \Sh_K(\SO(L_d)).
\]
We now recall some important properties.
\begin{thm}[\cite{pera15} Theorem 4.8, Corollary 4.15]
The map $\iota$ is \'etale.
The induced map $\iota : M^\circ_{2d} \to \Sh(\SO(L_d))$ is an open immersion.
\end{thm}
This also allows for comparisons of bundles, which will be a key ingredient for us to make statements about Galois representations.
\begin{prop}[\cite{pera15} 4.5]\label{prop:LpullbackP}
We have a compatible system of isometries of \'etale local systems $\alpha_\ell : \iota^\ast \LL_\ell \xrightarrow{\sim} \Pp^2_\ell$ on $\tilde M_{2d,\gamma}$.
\end{prop}

\subsection{The Kuga-Satake abelian variety}

The setup above leads to the following theorem.
\begin{thm}\label{thm:peramain}
Let $F$ be a field of characteristic $p \neq 2$ and $(X,\xi)$ a polarised K3 surface over $F$.
Passing to a finite separable extension if necessary, there exists an abelian variety $A/F$, called the \emph{Kuga-Satake abelian variety} such that:
\begin{enumerate}[(1)]
\item $A$ has a $\Z/2\Z$-grading and an action of the Clifford algebra $\Cl(L_d)$.
\item For each $\ell \neq p$ there is an isomorphism of $\Q_\ell$-vector spaces
\[
H^1(A_{F^s},\Q_\ell) \xrightarrow{~\sim~} \Cl(PH^2(X_{F^s},\Q_\ell)).
\]
\item The image of the Galois representation $\tilde\rho : G_F \to \GSp(H^1(A_{F^s},\Q_\ell))$ lands in
\[
\GSpin(PH^2(X_{F^s},\Q_\ell)) \subseteq \GSp(H^1(A_{F^s},\Q_\ell(1))).
\]

\item The image of the Galois representation $\rho: G_F \to \GL(PH^2(X_{F^s},\Q_\ell(1))$ lands in
\[
\SO(PH^2(X_{F^s},\Q_\ell(1))) \subseteq \GL(PH^2(X_{F^s},\Q_\ell(1))).
\]
Further, we have an equality $\ad \circ \tilde\rho = \rho$, where $\ad : \GSpin \to \SO$ is the adjoint map.
\end{enumerate}
\end{thm}
\begin{proof}
This appears as Theorem 4.17 in \cite{pera15}.
Properties (1) and (2) can be read off directly from that theorem.

Consider the diagram.
\[
\xymatrix{
& \Sh(\GSpin(L_d)) \ar[r] \ar[d] & \Sh(\GSp(\Cl(L_d),\psi_\delta)) \\
\tilde M_{2d,\gamma} \ar[r] & \Sh(\SO(L_d))
}
\]
We begin with a $F$-valued point of $M^\circ_{2d}$.
This lifts to a point $s \in \tilde M_{2d,\gamma}$.
After passing to a finite extension of $F$, we may lift $\iota(s) \in \Sh(\SO(L_d))$ to an $F$-point $\widetilde{\iota(s)} \in \Sh(\GSpin(L_d))$.
Taking the fibre of $\Aa^\KS$ at this point gives the abelian variety $A = \Aa^\KS_{\widetilde{\iota(s)}}/F$.

To study the Galois representation on $H^1(A_{F^s},\Q_\ell)$, we first introduce some notation.
For a group scheme $G/\Z$ we define a subgroup $G(N) \subseteq G(\hat \Z)$ by the exact sequence
\begin{equation}
1 \longrightarrow G(N) \longrightarrow G(\hat \Z) \longrightarrow G(\Z/N\Z) \longrightarrow 1.
\end{equation}
For a subgroup $H \subseteq G(\AA_f)$ we set $H(N) = H \cap G(N)$.

Let $K \subseteq \GSpin(L_d)(\AA_f)$ and $\K \subseteq \GSp(\Cl(L_d),\psi_\delta)(\AA_f)$ be compatible choices of open compact subgroup, sufficiently small such that $\Sh_\K(\GSp(\Cl(L_d),\psi_\delta))$ acquires its moduli description.
For $N$ such that $p\nmid N$, consider the diagram 
\begin{equation}
\xymatrix{
\Sh_{K(N)}(\GSpin(L_d)) \ar[r] \ar[d]^{K/K(N)} & \Sh_{\K(N)}(\GSp(\Cl(L_d),\psi_\delta)) \ar[d]^{\K/\K(N)} \\
\Sh_K(\GSpin(L_d)) \ar[r] & \Sh_\K(\GSp(\Cl(L_d),\psi_\delta))
}
\end{equation}
where the vertical maps are a $K/K(N)$ torsor and $\K/\K(N)$ torsor, respectively.
The fibre of such a vertical map above a point $s$ with corresponding abelian variety $A$ parameterises bases for the $N$-torsion $A[N]$ which extend the given $\K$-level structure.

Returning to the situation of the theorem, we consider the image of the $F$-point $\widetilde{\iota(s)}$ in $\Sh_K(\GSpin(L_d))$.
The action of the Galois group $G_F$ on $\Sh_{K(N)}(\GSpin(L_d))$ preserves the fibre above $\widetilde{\iota(s)}$ and corresponds to the action on $A[N]$.
In particular, the action on $G_F$ on $T_\ell A = \displaystyle\lim_{\substack{\leftarrow\\ n}} A[\ell^n]$ factors through $\displaystyle\lim_{\substack{\leftarrow\\ n}} K/K(\ell^n) \subseteq \GSpin(L_d)(\Z_\ell)$, giving property (3) above.

Finally, by Proposition \ref{prop:LpullbackP} we have an identification between $\rho^\ast \LL_\ell$ and $\Pp^2_\ell$ as bundles on $\tilde M_{2d,\gamma}$.
Further, we have $\LL_\ell = \PI_\ell (\Hh_\ell)$, in particular the finite cover $\Sh_{K_0(N)}(\SO(L_d)) \to \Sh_{K_0}(\SO(L_d))$ is a torsor for $K_0/K_0(N)$, identified with the image of $K/K(N)$ under the adjoint action, and property (4) follows.
\end{proof}

\section{The main theorem}\label{sec:thm}

We provide the remaining details of the proof of Theorem \ref{thm:K3descent}.
The supersingular case was discussed in Section \ref{sec:super}.
Below we cover the cases of $\ch K = 0$ or $X$ ordinary in positive characteristic.

\subsection{Passing the descent}

This begins with the following lemma.
\begin{lem}\label{lem:desc}
Let $K/k$ be a regular extension.
Let $X/K$ be a K3 surface and denote $\rho: G_K \to \GL\left(PH^2(X_{K^s}, \Q_\ell)\right)$ for $\ell \neq \ch K$.
Assume $\rho(\Gal(K^s / k^s K)) = 1$.
Passing to a finite extension if necessary, let $A/K$ be the Kuga-Satake abelian variety with $\tilde\rho : G_K \to \GL\left(H^1(A_{K^s},\Q_\ell)\right)$.
Then $\tilde\rho(\Gal(K^s / k^s K')) = 1$ for some $[K':K] \leq 2$.
\end{lem}
\begin{proof}
By Theorem \ref{thm:peramain} property (3), we know that $\tilde\rho(G_K) \subseteq \GSpin(L)(\Q_\ell)$.
Combining property (4) with the short exact sequence for $\ad$, we have the following commutative diagram.
\[
\xymatrix{
& \Gal(K^s / k^s K) \ar[ld] \ar[d]^{\tilde\rho} \ar[rd]^\rho & \\
\G_m(\Q_\ell) \ar[r] & \GSpin(L)(\Q_\ell) \ar[r]^{\ad} & \SO(L)(\Q_\ell)
}
\]
Hence we see that $\ad(\tilde\rho(\Gal(K^s / k^s K))) = \rho(\Gal(K^s / k^s K)) = 1$ and thus $\tilde\rho(\Gal(K^s / k^s K)) = \ker \ad = \G_m$.

Let $\mu : \GSp(H^1(A,\Q_\ell)) \to \Q_\ell^2$ be the symplectic similitude character.
Then we have $\mu \circ \rho = \chi_{\text{cycl}} : G_{K/k} \to \Q_\ell^\times$, the $\ell$-adic cyclotomic character.
However, by Lemma \ref{lem:spinsymp} we know that the symplectic similitude character coincides with the spinor norm $\nu$ on $\GSpin(L)(\Q_\ell) \subseteq \GSp(H^1(A,\Q_\ell))$.
Further, on $\ker \ad \cong \G_m$ we have $\nu|_{\ker \ad} : x \mapsto x^2$.
Thus we conclude for $g \in G_{K/k}$ that
\[
\tilde\rho(g)^2 = \nu(\tilde\rho(g)) = \mu(\tilde\rho(g)) = \chi_{\text{cycl}}(g).
\]
However, note that since $g \in G_{K/k}$ it fixes $k^s$ which necessarily contains all roots of unity, hence $\tilde\rho(g)^2 = \chi_{\text{cycl}}(g) = 1$.
Thus the action factors through an extension $[K':K] \leq 2$.
Then $\tilde\rho(\Gal(K^s / k^s K')) = 1$, as required.
\end{proof}
Hence, replacing $K$ with $K'$, we conclude that if $X$ satisfies $K/k$-Galois-$\ell$-descent, the same is true for its Kuga-Satake abelian variety $A$.
We can then apply Grothendieck's result, Theorem \ref{thm:groth}, to find an abelian variety $A_0/k$ and an isogeny $A_0 \otimes K \to A$.

\subsection{The cases of $X$ ordinary or $\ch K = 0$}

\begin{proof}[Proof of Theorem \ref{thm:K3descent}]
We begin with $X/K$ a K3 surface such that $\rho'(\Gal(K^s/k^sK)) = 1$, where $\rho' : G_K \to \Aut(H^2(X_{K^s},\Q_\ell))$.
Choose any polarisation $\xi$ on $X$, allowng us to define $PH^2$, and we can conclude $\rho(\Gal(K^s/k^sK)) = 1$, where $\rho : G_K \to \Aut(PH^2(X_{K^s},\Q_\ell))$.
Let $d = \deg(\xi)$.
Then $(X,\xi)$ gives a $K$-point $s \in M^\circ_{2d}(K)$.
As in Theorem \ref{thm:peramain} we have, after finite extension of $K$, a point $\widetilde{\iota(s)} \in \Sh(\GSpin(L_d))(K)$ and the Kuga-Satake abelian variety $A/K$ with $\tilde\rho : G_K \to \Aut(H^1(A_{K^s},\Q_\ell))$.
By Lemma \ref{lem:desc}, after quadratic extension of $K$ we have $\tilde\rho(\Gal(K^s/k^sK)) = 1$.
Thus, by Theorem \ref{thm:groth}, there exists an abelian variety $A_0/k$ and an isogeny $A_0 \otimes K \to A$.

Further, $A/K$ is equipped with extra data, including level structure and polarisation.
By independence of $\ell$ we have that $\Gal(K^s/k^sK)$ acts trivially on $T_\ell A$ for each $\ell$.
The level structure takes image in $\prod_{\ell' \neq p} T_{\ell'} A$, and is thus fixed by $\Gal(K^s/k^sK)$, thus descends to $k^s$.
The polarisation is a prime-to-$p$ isogeny $\lambda : A \to A^\vee$.
Factor $\ker \lambda$ into $\ell'$-primary components $(\ker \lambda)_{\ell'}$ for each $\ell' \neq p$.
Then $(\ker \lambda)_{\ell'}$ is a subscheme of $A[\ell'^n]$ for some $n$, which is fixed by $\Gal(K^s/k^sK)$.
Thus these are all defined over $k^s$ and thus the polarisation descends to $k^s$.

Now assume we are in case (1), i.e. $\ch K = 0$.
Then by Remark \ref{rmk:groth}, in fact we have an isomorphism $A_0 \otimes K \xrightarrow{\sim} A$.
To show the point $\widetilde{\iota(s)}$ is in fact a $k$-point, we first note that by choosing our level structure is sufficiently small, we find that the map $\Sh(\GSpin(L_d)) \to \Sh(\GSp((\Cl(L_d)))$ is an embedding.
Thus to argue descent of the point it is sufficient to consider the image, that is, we need to show the level structure and polarisation descend, as is shown above.
We therefore conclude $\widetilde{\iota(s)}$ was in fact a $k$-point, and after finite extension of $k$ we may assume the same of $s$.

Now assume we are in case (2), i.e. $X/K$ is ordinary.
By \cite{nygaard83} Prop 2.5 we have that $A/K$ is an ordinary abelian variety.
See also \cite{rizov10} Rmk 7.1.2 and \cite{pera15} 4.21.
Then by Remark \ref{rmk:groth}, in fact we have an isomorphism $A_0 \otimes K \xrightarrow{\sim} A$.
The argument above is again sufficient to conclude $\widetilde{\iota(s)}$ was in fact a $k$-point, and after finite extension of $k$ we may assume the same of $s$.

Thus, having taken suitable extensions of $k$ and $K$, $s$ is a $k$-point and we have $X_0/k$ with an isomorphism $X_0 \otimes K \xrightarrow{\sim} X$.
\end{proof}

\bibliographystyle{alpha}
\bibliography{galois-descent-bib}

\end{document}